\documentclass[sn-mathphys-num,pdflatex]{sn-jnl}
\usepackage[utf8]{inputenc}
\usepackage[T1]{fontenc}
\usepackage[english]{babel}

\usepackage{amsmath}
\usepackage{amsfonts}
\usepackage{amssymb}
\usepackage{amsthm}
\usepackage{mathtools}
\mathtoolsset{showonlyrefs}

\usepackage{graphicx}
\graphicspath{{Figures/}}
\usepackage{xcolor}

\usepackage{tikz,pgfplots}
\usepackage{scalerel}

\theoremstyle{plain}
\newtheorem{theorem}{Theorem}[section]
\newtheorem{lemma}[theorem]{Lemma}

\newtheorem{conjecture}[theorem]{Conjecture}
\newtheorem{remark}{Remark}[section]

\theoremstyle{definition}
\newtheorem{definition}{Definition}[section]

\newcommand{\hNinExponent}{{h_{\scaleto{N}{2.8pt}}}}
\newcommand{\aNinExponent}{{a_{\scaleto{N}{2.8pt}}}}

\newcommand{\abs}[1]{\lvert #1 \rvert}
\newcommand{\babs}[1]{\bigl\lvert #1 \bigr\rvert}

\newcommand{\dd}{\mathrm{d}}

\newcommand{\normI}[1]{\left\|#1\right\|_{\scriptscriptstyle 1}}
\newcommand{\normsup}[1]{\left\|#1\right\|_{\scriptscriptstyle\infty}}

\renewcommand{\emptyset}{\varnothing}
\newcommand{\setof}[2]{\{#1\,:\,#2\}}
\newcommand{\bsetof}[2]{\bigl\{#1\,:\,#2\bigr\}}

\newcommand{\defby}{\coloneqq}
\newcommand{\sfo}{\mathsf{o}}
\newcommand{\sfO}{\mathsf{O}}

\newcommand{\Z}{\mathbb{Z}}
\newcommand{\R}{\mathbb{R}}
\newcommand{\Rd}{\mathbb{R}^d}
\newcommand{\Zd}{\mathbb{Z}^d}

\newcommand{\Esp}{\mathsf{E}}
\newcommand{\given}{\,|\,}

\newcommand{\bbR}{\mathbb{R}}
\newcommand{\bbS}{\mathbb{S}}

\newcommand{\bbZ}{\mathbb{Z}}

\newcommand{\calA}{\mathcal{A}}

\newcommand{\calH}{\mathcal{H}}

\newcommand{\calM}{\mathcal{M}}

\newcommand{\calP}{\mathcal{P}}

\newcommand{\calV}{\mathcal{V}}
\newcommand{\calW}{\mathcal{W}}

\newcommand{\rmi}{\mathrm{i}}

\newcommand{\betac}{\beta_{\mathrm{\scriptscriptstyle c}}}

\newcommand{\GrCan}{\mathrm{GC}}
\newcommand{\Can}{\mathrm{C}}
\newcommand{\pressure}{\varphi}
\newcommand{\freeEne}{\mathrm{f}}

\begin{document}

\title{Fixed-magnetization Ising model with a slowly varying magnetic field}

\author[1]{\fnm{Yacine} \sur{Aoun}}\nomail
\author[2]{\fnm{S\'ebastien} \sur{Ott}}\nomail
\author[1]{\fnm{Yvan} \sur{Velenik}}\nomail

\affil[1]{\orgdiv{Department of mathematics}, \orgname{University of Geneva}, \orgaddress{\street{Rue du Conseil-G\'en\'eral 7--9}, \city{Gen\`eve}, \postcode{1205}, \country{Switzerland}}}
\affil[2]{\orgdiv{Department of mathematics}, \orgname{EPFL SB MATH}, \orgaddress{\street{B\^atiment MA}, \city{Lausanne}, \postcode{1015}, \country{Switzerland}}}

\abstract{%
	The motivation for this paper is the analysis of the fixed-density Ising lattice gas in the presence of a gravitational field.
	This is seen as a particular instance of an Ising model with a slowly varying magnetic field in the fixed magnetization ensemble.
	We first characterize the typical magnetization profiles in the regime in which the contribution of the magnetic field competes with the bulk energy term.
	We then discuss in more detail the particular case of a gravitational field and the arising interfacial phenomena.
	In particular, we identify the macroscopic profile and propose several conjectures concerning the interface appearing in the phase coexistence regime. The latter are supported by explicit computations in an effective model.
	Finally, we state some conjectures concerning equilibrium crystal shapes in the presence of a gravitational field, when the latter contributes to the energy only to surface order.
}


\maketitle


\section{Introduction}
The analysis of spatially inhomogeneous systems has a long history in theoretical physics and chemistry, see for instance~\cite{Rogiers-1993,Stillinger-1962,Hart-1959,Meister-1985}. In contrast, the specific equilibrium properties of spatially inhomogeneous lattice systems have been surprisingly little studied in the mathematical physics literature compared to the homogeneous ones. The inhomogeneous nature stems from the interaction~\cite{Bovier-2006} or from the external potential~\cite{Millard-1972,Presutti-1973,Simmons+Garrod-1973,Bissacot+Cassandro+Cioletti+Presutti-2015}. The present paper is the first in a planned series devoted to various aspects of the latter. 

In the present work, we consider an Ising model, in a fixed-magnetization ensemble, that is subject to an inhomogeneous magnetic field that varies sufficiently slowly (in a technical sense made precise below). We derive, using thermodynamic arguments (similar to what has been done in~\cite{Presutti-1973,Simmons+Garrod-1973}) and large deviations arguments, the geometry of the typical magnetization profile. We then focus on a specific example: the Ising lattice gas (at fixed density) in the presence of a slowly varying gravitational field. In this case, we make several conjectures concerning the behavior of the typical magnetization and corresponding interfaces by studying a much simpler effective model that we heuristically motivate.

The present paper is intended to have a rather expository character, presenting basic results with soft arguments and making predictions about various relevant phenomena that arise. It will be completed by a series of papers investigating in more detail these (and related) problems.

Let us briefly comment on two related rather recent works. In~\cite{Grill+Tutschka}, an Ising lattice gas in a gravitational field is also considered. There are two main simplifications with respect to the current work. First, they work in the (simpler) grand canonical ensemble. Second, they consider a scaling of the gravitational field with the size of the domain such that its contribution to the energy becomes much larger than the interaction term. This leads to a behavior similar to the model without interactions; in particular, there is no qualitative change of behavior as the temperature is varied. In~\cite{Montino+Soprano-Loto+Tsagkarogiannis}, an Ising model with spatially modulated magnetization at large scales is introduced. The desired local magnetization is enforced by adding a suitable Kac-type quadratic interaction. The main concern is the determination of the free energy and pressure of the resulting model, however the geometrical properties of the resulting magnetization profiles, in particular the properties of the interfaces, are not investigated.

The paper is organized as follows: in Section~\ref{sec:definitions}, we introduce the Ising model with a slowly varying magnetic field. In Section~\ref{sec:thermo_quantities_result}, we recall basic results on thermodynamical quantities that are used in our arguments. In Section~\ref{sec:TypicalProfiles}, we analyze the geometry of typical magnetization profiles. Finally, in Section~\ref{sec:grav_field}, we study the particular example of a gravitational field. We discuss two relevant scaling for this problem. In the first case, the gravitational field has an impact on the thermodynamical properties, resulting in particular in a height-dependent density profile, possibly including an interface. We make a precise conjecture about the latter's scaling limit, which we motivate by establishing the corresponding claim in a simple effective model. Finally, we briefly discuss the second relevant scaling, in which the gravitational field only impacts surface phenomena. In particular, we describe its effect on the macroscopic geometry of phase separation, postponing a detailed analysis of this problem to future work.

\section{Ising model in a slowly varying magnetic field}\label{sec:definitions}

Let \(\Lambda_N \defby \{0,\dots, N-1\}^d\) and set \(\Omega_N \defby \{-1,1\}^{\Lambda_N}\). Consider the Ising model in the box \(\Lambda_N\) with Hamiltonian
\[
	H_0(\sigma) \defby -\sum_{\substack{\{i,j\}\subset \Lambda_N\\ \normI{i-j} = 1}} \sigma_i\sigma_j,\quad \forall\sigma\in\Omega_N.
\]
In this work, our goal is to analyze the effect of a slowly varying inhomogeneous magnetic field on the model. Namely
we consider a sequence \(h_N:\Lambda_N\to \R\) of magnetic fields. We say that \(h_N\) is \emph{slowly varying} if there are two sequences \(a_N\in \Z_+\) and \(b_N\in \R_+\) such that
\begin{itemize}
	\item \(a_N\xrightarrow{N\to \infty} \infty\) and \(a_N/N\xrightarrow{N\to \infty} 0\),
	\item \(b_N\xrightarrow{N\to \infty} 0\),
	\item for any \(N\) and any \(i,j\in \Lambda_N\),
	\begin{equation}\label{eq:SlowVariation}
		\normsup{i-j}<a_N \implies |h_N(i)- h_N(j)|<b_N.
	\end{equation}
	\item \(\sup_{N\geq 1}\sup_{i\in\Lambda_N} \abs{h_N(i)} < \infty\).
\end{itemize}
A particularly relevant example is that of a gravitational field \(h_N(i) = g_N i_d\), with an intensity \(g_N = \sfo_{N\to\infty}(1)\). The physically relevant scaling of \(g_N\) with \(N\) depends on the specific problem analyzed and will be discussed later in the paper.

\smallskip
To simplify the exposition, we will assume that \(N/a_N\) is an integer (the adaptation to deal with other cases is straightforward). Denote \(\Gamma_N = \{0,a_N,2a_N,\dots,N-a_N\}^d\).
\begin{remark}
	The conditions above imply that the limiting function \(h:[0,1)^d \to \bbR\), \(h(x) \defby \lim_{N\to\infty} h_N([Nx])\) is continuous (here, \([y]_i = \lfloor y_i \rfloor\) for any \(y\in\Rd\)).
	It would be possible to allow discontinuities, as long as their measure in the continuum limit vanishes.
\end{remark}

The Hamiltonian of our model then takes the form
\[
	H_{h_N}(\sigma) \defby H_0(\sigma) - \sum_{i\in\Lambda_N} h_N(i) \sigma_i,\quad \forall\sigma\in\Omega_N.
\]

Let us denote by \(M_N(\sigma) \defby \sum_{i\in\Lambda_N} \sigma_i\) the total magnetization in the box \(\Lambda_N\) and by
\(\textsf{Mag}_N \defby \setof{M_N(\sigma)}{\sigma\in\Omega_N}\)
the set of all possible values of the magnetization in the box \(\Lambda_N\).

Fix
\(M\in\mathsf{Mag}_N\)
and let
\(\Omega_{N,M} \defby \setof{\sigma\in\Omega_N}{M_N(\sigma)=M}\)
.
The canonical ensemble at inverse temperature \(\beta\in\R_{\geq 0}\) associated to the magnetization 
\(M\) 
is described by the probability measure on
\(\Omega_{N,M}\)
defined by
\[
\mu_{N,\beta,\hNinExponent,M}(\sigma) \defby \frac{1}{Z_{N,\beta,\hNinExponent,M}} e^{-\beta H_{\hNinExponent}(\sigma)},\qquad Z_{N,\beta,\hNinExponent,M} \defby \sum_{\sigma\in\Omega_{N,M}} e^{-\beta H_{\hNinExponent}(\sigma)}.
\]

To simplify notations, we will make the following abuse of notation: for \(m\in [-1,1]\), we will write \(Z_{N,\beta,\hNinExponent,m}\) (and \(Z_{N,\beta,m}^{\Can}\), see below) for \(Z_{N,\beta,\hNinExponent,K_N(m)}\) (and \(Z_{N,\beta,K_N(m)}^{\Can}\) respectively) where \(K_N(m) \in \mathsf{Mag}_N\) is an (arbitrary) sequence such that \(N^{-d}K_N(m)\to m\). As the results do not depend on the sequence, we allow ourselves this slight abuse and hope it will not confuse the reader. The same will be used for the corresponding measures.

\section{Thermodynamic quantities in homogeneous systems}\label{sec:thermo_quantities_result}

We will use some classical results about large deviations and ensemble equivalence for the Ising model with a constant magnetic field.
Define the Grand Canonical and Canonical partition functions
\begin{gather}
	Z_{N,\beta,h}^{\GrCan} \defby \sum_{\sigma\in\Omega_N} \exp\bigl[ -\beta \bigl( H_0(\sigma) - h M_N(\sigma) \bigr) \bigr],\\
	Z_{N,\beta,M}^{\Can} \defby \sum_{\sigma\in\Omega_{N,M}} e^{-\beta H_0(\sigma) }.
\end{gather}
We will need a number of standard results.

\begin{theorem}\label{thm:ThermoPot}
	Let \(K_N\in \textsf{Mag}_N\) be any sequence such that \(N^{-d}K_N\to m\).\\
	The \emph{pressure}
	\begin{equation}
		\label{eq:pressure}
		\pressure_{\beta}(h) \defby \lim_{N\to \infty} \frac{1}{\beta N^d} \log Z_{N,\beta,h}^{\GrCan},
	\end{equation}
	and the \emph{free energy}
	\begin{equation}
		\label{eq:free_energy}
		\freeEne_{\beta}(m) \defby -\lim_{N\to \infty} \frac{1}{\beta N^d} \log Z_{N,\beta,K_N}^{\Can},
	\end{equation}
	are well defined for all \(\beta\geq 0, h\in\bbR, m\in [-1,1]\) and are convex functions of their argument. Moreover, for any pair of sequences \(K_N^-\leq K_N^+\) such that \(N^{-d} K_N^{\pm}\to m\),
	\[
		-\lim_{N\to \infty} \frac{1}{\beta N^d} \log \sum_{K=K_N^-}^{K_N^+}
		Z_{N,\beta,K}^{\Can}
		= \freeEne_{\beta}(m).
	\]
\end{theorem}
\begin{proof}
	See for instance~\cite[Theorems~4.5 and~4.11]{Friedli+Velenik-2017}.
\end{proof}
\begin{theorem}\label{thm:PhaseTransition}
		For all \(\beta\leq\betac\), \(\pressure_{\beta}\) is differentiable on \(\R\) and its derivative \(m_\beta(h) \defby \frac{\mathrm{d}}{\mathrm{d}h}\pressure_\beta(h)\) is (strictly) increasing and continuous on \(\R\).\\
		For all \(\beta>\betac\), \(\pressure_{\beta}\) is differentiable on \(\R^*\defby\bbR\setminus\{0\}\)  and its derivative \(m_\beta(h) \defby \frac{\mathrm{d}}{\mathrm{d}h}\pressure_\beta(h)\) is (strictly) increasing and continuous on \(\R^*\). Moreover,
		\[
			m_\beta^* \defby \lim_{h\to 0^+} m_{\beta}(h) = - \lim_{h\to 0^-} m_{\beta}(h) > 0.
		\]
\end{theorem}
\begin{proof}
	For the statement about the case \(\beta=\betac\) and \(h=0\), see~\cite{Aizenman+Duminil-Copin+Sidoravicius-2015}. For the remaining statements, see, for instance, Chapter~3 in \cite{Friedli+Velenik-2017}. In fact, when \(h\neq 0\) or \(\beta<\betac\), \(\pressure_\beta\) is analytic; see~\cite{Ott-2020}.
\end{proof}
\begin{theorem}[Equivalence of ensembles]\label{thm:EquivalenceEnsembles}
	For any \(\beta \geq 0\), \(\pressure_{\beta}\) and \(\freeEne_{\beta}\) are related by
	\begin{alignat}{2}
			\freeEne_{\beta}(m) &= \sup_{h\in \R} \bigl(hm - \pressure_{\beta}(h)\bigr),\quad &&\forall m\in [-1,1],\\
			\pressure_{\beta}(h) &= \sup_{m\in [-1,1]}\bigl(m h - \freeEne_{\beta}(m)\bigr),\quad &&\forall h\in\R.
			\label{eq:LegendrePressure}
	\end{alignat}
	Moreover, whenever the pressure is differentiable (that is, when \(h\neq 0\) or \({\beta\leq\betac}\)), the supremum in~\eqref{eq:LegendrePressure} is attained at \(m=m_\beta(h)\). When \(\beta>\betac\) and \(h=0\), the set of points at which the supremum is attained is the phase-coexistence interval \([-m_\beta^*,m_\beta^*]\).
\end{theorem}
\begin{proof}
	See, for instance, Theorem~4.13 and Section~4.8.3 in~\cite{Friedli+Velenik-2017}.
\end{proof}

\section{Typical mesoscopic magnetization profiles}
\label{sec:TypicalProfiles}

Our main results will be stated in terms of mesoscopic magnetization profiles, which we define now.
\begin{definition}
	A \emph{mesoscopic magnetization profile} (or just \emph{profile} for simplicity) in \(\Lambda_N\) is a function \(q_N:\Gamma_N \to [-1,1]\). The profile \(q_N\) is said to be \emph{compatible} with the magnetization density \(m\) if
	\[
		\frac{1}{|\Gamma_N|}\sum_{x\in \Gamma_N} q_N(x) = m.
	\]
	We denote this compatibility relation by \(q_N\sim m\) and write \(\calP_N(m) = \{q_N:\Gamma_N \to [-1,1]:\ q_N\sim m \}\).
\end{definition}
\begin{figure}[t]
	\centering
	\begin{tikzpicture}[scale=0.25]
		\fill[lightgray] (15.5,15.5) rectangle ++ (5,5);
		\foreach \x in {1,...,33} {
			\draw (1,\x) -- (33,\x);
			\draw (\x,1) -- (\x,33);
		}
		\foreach \x in {1,...,33} {
			\foreach \y in {1,...,33} {
				\filldraw[fill=white] (\x,\y) circle(7pt);
			}
		}
		\foreach \x in {1,6,11,...,33} {
			\foreach \y in {1,6,11,...,33} {
				\fill[black] (\x,\y) circle(7pt);
			}
		}
		\draw[thick] (15.5,0.5) -- (15.5,0) -- (20.5,0) -- (20.5,0.5);
		\draw (18.3,0) node[below] {$a_N$};
	\end{tikzpicture}
	\caption{The box \(\Lambda_N\) and the subset \(\Gamma_N\) (black vertices). One of the cells is shaded; it is associated to the vertex of \(\Gamma_N\) located at its bottom left corner.}
	\label{fig:BoxGammaCell}
\end{figure}
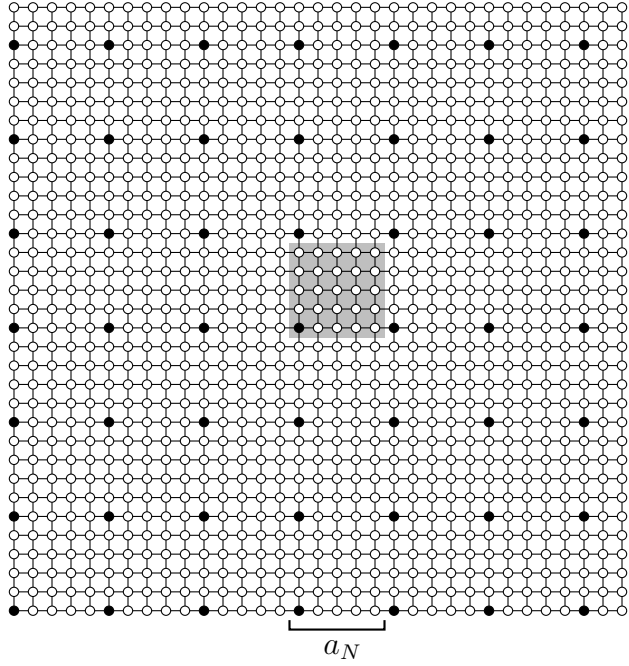
The first step in our analysis is the following observation.
\begin{lemma}\label{lem:LowerBound}
	Let \(h_N\) be slowly varying and \(m\in [-1,1]\). Then,
	\[
		\frac{1}{\beta N^d} \log Z_{N,\beta,h_N,m} \geq  \sup_{q_N\in \calP_N(m)} \frac{1}{|\Gamma_N|}\sum_{x\in \Gamma_N} \bigl(h_N(x)q_N(x) - \freeEne_\beta(q_N(x))\bigr) + \sfo_N(1).
	\]
\end{lemma}
\begin{proof}
	Let \(q_N\in \calP_N(m)\). Let \(\tilde{q}_N\in \calP_N(m)\) be such that, for all \(x\in\Gamma_N\),
	\[
		\bsetof{\sigma\in\Omega_{a_N}}{M_{a_N}(\sigma) = \tilde{q}_N(x) a_N^d} \neq \emptyset
		\qquad\text{and}\qquad
		\abs{q_N(x) - \tilde{q}_N(x)} a_N^d \leq 2.
	\]
	(The existence of such a profile is easy to establish: start by setting \(\check{q}_N(x) = \lfloor q_N(x) a_N^d \rfloor\) for each \(x\in\Gamma_N\). By construction, \(\sum_{x\in\Gamma_N} \check{q}_N(x) \in {(mN^d - 2\abs{\Gamma_N}, mN^d]}\). One then simply flip one minus spin in different cells until the resulting total magnetization is (the implicit allowed total magnetization approximating) \(mN^d\).)
	
	We get a lower bound on \(Z_{N,\beta,h_N,m}\) by restricting the summation to configurations \(\sigma\sim \tilde{q}_N\), by which we mean that
	\[
		\sum_{i\in x+\Lambda_{a_N}} \sigma_i = \tilde{q}_N(x) a_N^d.
	\]
	Note that, for any \(x\in \Gamma_N\), \(i\in x+\Lambda_{a_N} \) implies \(|h_N(i)-h_N(x)|<b_N\), as \(\normsup{i-j}< a_N\) implies \(|h_N(i)-h_N(j)|< b_N\) by~\eqref{eq:SlowVariation}. This implies that, for any \(\sigma\sim \tilde{q}_N\),
	\begin{equation}
		\prod_{i\in x+\Lambda_{a_N}} e^{\beta h_N(i)\sigma_i} \geq e^{-\beta b_N a_N^d} e^{\beta h_N(x) \tilde{q}_N(x) a_N^d}.
	\end{equation}
	This yields
	\[
		Z_{N,\beta,h_N,m} \geq e^{-c_d\beta |\Gamma_N|a_N^{d-1}} e^{-\beta b_N N^d} \prod_{x\in \Gamma_N} Z_{a_N,\beta,\tilde{q}_N(x)}^{\Can} e^{\beta h_N(x) \tilde{q}_N(x) a_N^d}.
	\]
	Using \(\tilde{q}_N a_N^d \leq q_N a_N^d + 2\), taking the log, dividing by \(\beta N^d\) and using the uniform boundedness of \(h_N\) gives the result by Theorem~\ref{thm:ThermoPot}.
\end{proof}
Lemma~\ref{lem:LowerBound} gives rise to a variational problem, namely maximize
\[
	\Psi_N(q) \defby \frac{1}{|\Gamma_N|}\sum_{x\in\Gamma_N} \bigl(h_N(x)q(x) - \freeEne_\beta(q(x))\bigr)
\]
over all \(q\in \calP_N(m)\). It turns out that the latter can be solved explicitly.
To avoid some pathologies (and because the cases \(m=\pm 1\) are trivial), we assume from now on that \(m\in (-1,1)\).

Let us introduce the following notations:
\[
	\begin{alignedat}{2}
		\calV(h_N) &\defby \{h_N(x),\; x\in\Gamma_N\}, \quad &&\forall h_N:\Gamma_N\to\bbR,\\
		\Gamma_N[h] &\defby \setof{x\in\Gamma_N}{h_N(x) = h}, \quad &&\forall h\in\bbR.
	\end{alignedat}
\]
(Of course, \(\Gamma_N[h]=\emptyset\) if \(h\notin\calV(h_N)\).)

\medskip
We are now going to construct explicitly a maximizer \(q^*_N\) of \(\Psi_N\) on \(\calP_N(m)\).
The idea is very simple. By Theorem~\ref{thm:EquivalenceEnsembles}, the maximum of \(h_N(x)q(x) - \freeEne_\beta(q(x))\) is reached when \(q(x) = m_\beta(h_N(x))\). However, in general, \(\frac{1}{\abs{\Gamma_N}}\sum_{x\in\Gamma_N} q(x) \neq m\), so that \(q\notin \calP_N(m)\). The solution to this problem is to observe that, since the overall magnetization is fixed, adding a constant \(\bar{h}\) to the magnetic field has only a trivial effect on the energy, namely shifting it by \(-m\bar{h}\). We are thus led to finding \(\bar h\) such that the function \(x\mapsto q(x) = m_\beta(h_N(x)-\bar h)\) belongs to \(\calP_N(m)\). Let us now make this procedure precise.

\medskip
For \(h\in\bbR\setminus\calV(h_N)\), we write \(\calM(h) \defby \abs{\Gamma_N}^{-1}\sum_{x\in\Gamma_N} m_\beta(h_N(x) - h)\) and set
\[
	\bar{h} \defby \sup\bsetof{h\in\bbR\setminus\calV(h_N)}{\calM(h)\geq m}.
\]
Note that \(\bar{h}\in\bbR\), since \(\lim_{h\to -\infty} \calM(h) = 1 > m > -1 = \lim_{h\to +\infty} \calM(h)\). 

Since, by Theorem~\ref{thm:PhaseTransition}, the function \(h\mapsto m_\beta(h)\) is (strictly) increasing on \(\bbR\) and continuous on \(\bbR^*\), it follows that the function \(h \mapsto \calM(h)\) is (strictly) decreasing and continuous on \(\bbR\setminus\calV(h_N)\). Furthermore,
\begin{equation}\label{eq:MagJump}
	\forall h\in\calV(h_N),\quad \lim_{h' \to h^-}\calM(h') - \lim_{h' \to h^+}\calM(h') = 2m_\beta^*\frac{\abs{\Gamma_N[h]}}{\abs{\Gamma_N}} ,
\end{equation}
since, when \(h\in\calV(h_N)\), the magnetic field \(h_N-h\) vanishes in all cells of \(\Gamma_N[h]\), and the optimal value for the magnetization in such cells must belong to the coexistence plateau \([-m^*_\beta,m^*_\beta]\). This induces a jump of \(2m^*_\beta\) for the magnetization in each of these cells when a value \(h\in\calV(h_N)\) is crossed.

In particular, \(\calM(\bar{h}) = m\) when \(\bar{h}\notin\calV(\Gamma_N)\). However, when \(\bar{h}\in\calV(\Gamma_N)\), the function \(h\mapsto\calM(\bar{h})\) is discontinuous at \(\bar{h}\) and the value \(m\) belongs to the gap. We solve this difficulty by fixing the value of the magnetization in the cells of \(\Gamma_N[\bar{h}]\) to the required value to compensate the difference.

Namely, we define the mesoscopic magnetization profile \(q^*_N:\Gamma_N\to [-1,1]\) by
\[
q^*_N(x) \defby
\begin{cases*}
	m_\beta(h_N(x)-\bar{h})											&	if \(x\notin\Gamma_N[\bar{h}]\),\\
	m_\beta^* + \frac{\abs{\Gamma_N}}{\abs{\Gamma_N[\bar{h}]}}(m - \calM(\bar{h}))	& if \(x\in\Gamma_N[\bar{h}]\).
\end{cases*}
\]
Note that this is well defined, since~\eqref{eq:MagJump} implies that \(m_\beta^* + \frac{\abs{\Gamma_N}}{\abs{\Gamma_N[\bar{h}]}}(m - \calM(\bar{h})) \in [-m_\beta^*,m_\beta^*]\). This should not be surprising, since, as explained above, \(h_N(x)-\bar{h}=0\) when \(x\in\Gamma_N[\bar{h}]\) and the system should therefore be in the phase coexistence regime (in the corresponding cell).

\medskip
It is straightforward to check that \(q^*_N\in \calP_N(m)\). Let us ensure that \(q^*_N\) is indeed a maximizer of \(\Psi_N\).
\begin{lemma}\label{lem:OptimalProfile}
	Let \(m\in (-1,1)\).
	\(q^*_N\) is a maximizer of \(\Psi_N\) on \(\calP_N(m)\). Moreover, all other maximizers in \(\calP_N(m)\) differ from \(q^*_N\) only on \(\Gamma_N[\bar{h}]\).
\end{lemma}
\begin{proof}
	Let \(q_N\in \calP_N(m)\). Then,
	\begin{align*}
		\Psi_N(q_N)
		&=
		\bar{h} m + \frac{1}{|\Gamma_N|}\sum_{x\in \Gamma_N} \bigl((h_N(x)-\bar{h})q_N(x) - \freeEne_\beta(q_N(x))\bigr) \\
		&\leq
		\bar{h} m + \frac{1}{|\Gamma_N|}\sum_{x\in \Gamma_N} \pressure_{\beta}(h_N(x)-\bar{h}) \\
		&=
		\bar{h} m + \frac{1}{|\Gamma_N|}\sum_{x\in \Gamma_N} \bigl((h_N(x)-\bar{h})q^*_N(x) - \freeEne_\beta(q^*_N(x))\bigr) \\
		&=
		\Psi_N(q^*_N),
	\end{align*}
	where the inequality and the identity in the third line follow from Theorem~\ref{thm:EquivalenceEnsembles} (since \(q^*_N(x) = m_\beta(h_N(x)-\bar{h})\) when \(x\notin\Gamma_N[\bar{h}]\) and \(q^*_N(x)\in [-m_\beta^*,m_\beta^*]\) when \(x\in\Gamma_N[\bar{h}]\)).

	\smallskip
	Let us now consider another maximizer \(q'_N\in \calP_N(m)\). The considerations above show that changing the value of \(q^*_N(x)\) at any \(x\notin\Gamma_N[\bar{h}]\) necessarily (strictly) decreases \(\Psi_N\). Therefore, the only \(x\in\Gamma_N\) at which \(q'_N(x)\) can differ from \(q^*_N(x)\) belong to \(\Gamma_N[\bar{h}]\). (Note that, in general, \(\Psi_N\) indeed admits several maximizers: they correspond to all possible choices of \(q'_N(x)\), \(x\in\Gamma_N[\bar{h}]\), satisfying both \(\sum_{x\in\Gamma_N[\bar{h}]} q'_N(x) = \sum_{x\in\Gamma_N[\bar{h}]} q^*_N(x)\) and \(q'_N(x)\in [-m_\beta^*, m_\beta^*]\) for all \(x\in\Gamma_N[\bar{h}]\).)
\end{proof}

Lemmas~\ref{lem:LowerBound} and~\ref{lem:OptimalProfile} provide a lower bound on the partition function of the system that can be used to prove concentration of typical profiles on \(q^*_N\) (in the \(L^1\) norm). Let us write \(\mathsf{m}_N(x) \defby a_N^{-d}\sum_{i\in x+\Lambda_{\aNinExponent}} \sigma_i\) for the empirical magnetization density in the cell indexed by \(x\in\Gamma_N\).

\begin{theorem}
	Let \(m\in (-1,1)\). Consider a sequence of slowly varying magnetic fields \(h_N\) such that \(\lim_{N\to\infty}\abs{\Gamma_N[h]}/\abs{\Gamma_N} = 0\) for all \(h\in\calV(h_N)\). Then, for all \(\epsilon>0\), there exist \(c>0\) and \(N_0\) such that
	\[
		\mu_{N,\beta,h_N,m}\Bigl( \sum_{x\in\Gamma_N} \babs{\mathsf{m}_N(x) - q^*_N(x)} \geq \epsilon \abs{\Gamma_N} \Bigr) \leq e^{-c \beta N^d},
	\]
	for all \(N\geq N_0\).
\end{theorem}
\begin{proof}
	Let us denote by \(\calA \defby \setof{\sigma\in\{-1,1\}^{\Lambda_N}}{\sum_{x\in\Gamma_N} \babs{\mathsf{m}_N(x) - q^*_N(x)} \geq \epsilon \abs{\Gamma_N}}\) the event under consideration.
	When \(\calA\) occurs, there exist at least \(\frac{1}{4}\epsilon\abs{\Gamma_N}\) vertices \(x\in\Gamma_N\) such that \(\abs{\mathsf{m}_N(x) - q^*_N(x)} \geq \frac{1}{2}\epsilon\). Indeed, were it not the case, we would have
	\[
		\sum_{x\in\Gamma_N} \babs{q^*_N(x) - \mathsf{m}_N(x)} < 2 \cdot \tfrac{1}{4}\epsilon\abs{\Gamma_N} + \tfrac{1}{2}\epsilon \cdot \abs{\Gamma_N} = \epsilon \abs{\Gamma_N}.
	\]
	A slight nuisance is due to the possible presence of cells belonging to \(\Gamma_N[\bar{h}]\), inside which we do not have good concentration properties (in these cells, the system is in the phase coexistence regime). However, these cells cannot be too numerous. Indeed, by assumption, \(\lim_{N\to\infty}\abs{\Gamma_N[\bar{h}]}/\abs{\Gamma_N} = 0\). We can therefore choose \(N_0\) large enough to ensure that \(\abs{\Gamma_N[\bar{h}]} < \frac{1}{8}\epsilon\abs{\Gamma_N}\).
	This guarantees that there are at least \(\frac{1}{8}\epsilon\abs{\Gamma_N}\) vertices \(x\in\Gamma_N\setminus\Gamma_N[\bar{h}]\) such that \(\abs{\mathsf{m}_N(x) - q^*_N(x)} \geq \frac{1}{2}\epsilon\).
	We thus have
	\[
		\mu_{N,\beta,h_N,m}(\calA)
		\leq
		\sum_{\substack{b\subset\Gamma_N\setminus\Gamma_N[\bar{h}]\\\abs{b}\geq \frac18\epsilon\abs{\Gamma_N}}}
		\mu_{N,\beta,h_N,m}\bigl(\forall x\in b,\, \abs{\mathsf{m}_N(x) - q^*_N(x)} \geq \tfrac{1}{2}\epsilon \bigr).
	\]
	As follows from the proof of Lemma~\ref{lem:OptimalProfile} and strict convexity of \(f_\beta\) outside the coexistence interval, there exists \(c=c(\epsilon) >0\) such that
	\begin{equation}\label{eq:DecreasePressureBadBoxes}
		\abs{q_N(x) - q^*_N(x)} \geq \tfrac{1}{2}\epsilon \implies h_N(x)q_N(x) - \freeEne_\beta(q_N(x)) \leq h_N(x)q^*_N(x) - \freeEne_\beta(q^*_N(x)) - c.
	\end{equation}
	Proceeding as in the proof of Lemma~\ref{lem:LowerBound}, we get, for any admissible (that is, that can be realized as an empirical profile) \(q_N\in \calP_N(m)\) such that \(\abs{q_N(x) - q^*_N(x)} \geq \tfrac{1}{2}\epsilon\) for all \(x\in b\),
	\begin{align*}
		\mu_{N,\beta,h_N,m}(\mathsf{m}_{\color{red}N}=q_N)
		&\leq \frac{1}{Z_{N,\beta,h_N,m}} e^{\beta \Psi_N(q_N) N^d + \sfo(N^d)} \\
		&\leq e^{\beta (\Psi_N(q_N)-\Psi_N(q^*_N)) N^d + \sfo(N^d)} \leq e^{-\frac12\beta c a_N^d \abs{b}},
	\end{align*}
	where we used Lemmas~\ref{lem:LowerBound} and~\ref{lem:OptimalProfile} for the second inequality, and~\eqref{eq:DecreasePressureBadBoxes} for the last one (since \eqref{eq:DecreasePressureBadBoxes} implies that \(\Psi_N(q_N)\leq\Psi_N(q^*_N) - ca_N^d\abs{b}\)).
	
	Since the number of admissible profiles \(q\) is equal to \((a_N^d+1)^{\abs{\Gamma_N}} = e^{\sfo(N^d)}\), we conclude that
	\[
		\mu_{N,\beta,h_N,m}(\calA)
		\leq
		e^{\sfo(N^d)}\!\!\!\!\!
		\sum_{\substack{b\subset\Gamma_N\\\abs{b}\geq \frac14\epsilon\abs{\Gamma_N}}} e^{-\frac12\beta c a_N^d \abs{b}}
		\leq
		e^{-\frac{1}{10}\beta c\epsilon N^d}. \qedhere
	\]
\end{proof}

\section{The case of a gravitational field}\label{sec:grav_field}

In this section, we consider one particular instance of the general framework described in Section~\ref{sec:TypicalProfiles}: the case of a linearly growing magnetic field. Although we will stick to the magnetic language for simplicity, the physical interpretation is more natural in the lattice gas interpretation, since this linearly growing magnetic field can be interpreted as a gravitational field acting on the particles of the gas. 

Let \(m\in (-1,1)\), let \(g_N\) be a sequence of positive real numbers and let \(m_N\in\mathsf{Mag}_N\) be a sequence converging to \(m\). In this section, we consider a magnetic field given by
\begin{equation}\label{eq:hNGrav}
	h_N(i) \defby g_N i_d, \quad \forall i=(i_1,\dots,i_d)\in\Lambda_N.
\end{equation}
For the reason explained above, we will refer to this particular form of the magnetic field as a gravitational field.

\medskip
We are going to consider two different scalings for the intensity \(g_N\) of the gravitational field. We will be mostly interested in the case \(g_N = g/N\), which is the one relevant if one wants the field to affect the local thermodynamic properties of the system (in particular the value of local magnetization density as a function of the height); this will be discussed in Section~\ref{sec:GFbulk}. We will also briefly discuss the physically very relevant case of a gravitational field of the form \(g_N = g/N^2\), for which the field has no effect on the local thermodynamic properties, but affects the macroscopic geometry of phase separation; this will be briefly discussed in Section~\ref{sec:GFsurface}, although a detailed analysis is postponed to a future work.

\medskip
For simplicity of exposition, we will describe only the limiting magnetization profile, that is, we will consider the function \(q^*:[0,1)^d\to [0,1]\) given by
\[
	q^*(x) \defby \lim_{N\to\infty} q_N^*(\lfloor Nx/a_N \rfloor a_N), \quad \forall x\in [0,1)^d.
\]

\subsection{\(1/N\) scaling: density profiles}
\label{sec:GFbulk}

In this subsection, we focus on the case \(g_N=g/N\) with \(g<0\). It is straightforward to check that this indeed corresponds to a slowly varying magnetic field in the sense of the previous sections, with \(b_N=g a_N/N\) and \(a_N\) an arbitrary sequence of positive numbers that tends to infinity in such a way that \(a_N=\sfo(N)\).

Note that, with this particular normalization, the contribution of the gravitational term is of order \(N^d\) and thus competes with the interaction energy. This is thus the relevant scaling if one wants the gravitational field to have a nontrivial effect on the local density as the height changes. Two typical configurations are given in Figure~\ref{fig:GFbulk-1000}, one above the critical temperature, one below. Although both clearly show a decrease of the density with the height, the second one exhibits a much more discontinuous behavior, manifested by a clear interface separating a liquid phase from a gas phase. This can be easily understood, at the level of mesoscopic density profiles, using the results of Section~\ref{sec:TypicalProfiles}.

\begin{figure}[ht]
	\centering
	\includegraphics[height=6cm]{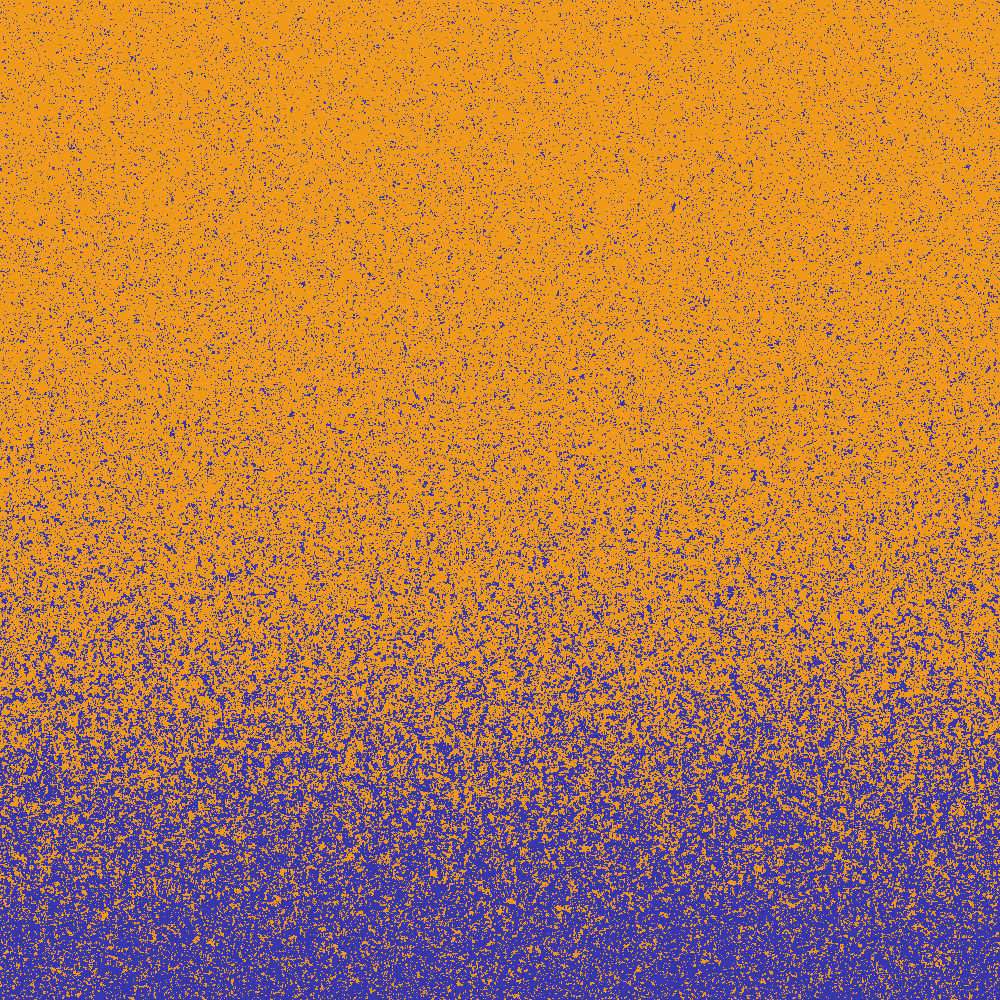}
	\hspace{5mm}
	\includegraphics[height=6cm]{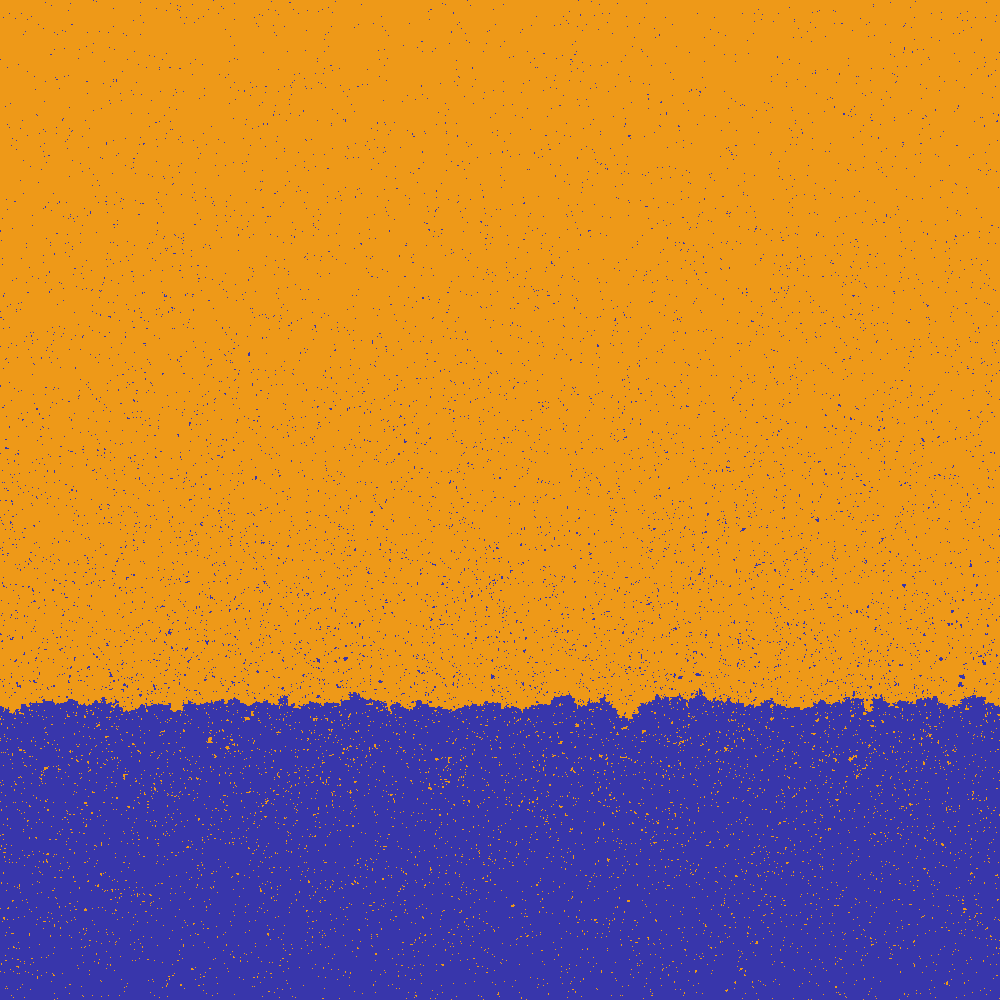}
	\caption{Typical configurations of the two-dimensional model with \(N=1000\), \(g=-1\), \(m=0.4\) and \(\beta=0.3\) (left) and \(0.5\) (right) (the critical inverse temperature of the zero-field planar Ising model is \(\betac=\frac12\operatorname{arsinh}(1)\cong 0.44\)). Here, \(+\) spins are in orange and \(-\) spins in blue.}
	\label{fig:GFbulk-1000}
\end{figure}

\subsubsection{The magnetization profile in the continuum limit}
Since \(h_N(Nx) = g x_d\), for any \(x=(x_1,\dots,x_d)\in [0,1)^d\),
\begin{align*}
	\lim_{N\to\infty} \abs{\Gamma_N}^{-1} \sum_{x\in\Gamma_N} m_\beta(h_N(x) - h) &= \int_{[0,1]^d} m_\beta(gx_d - h)\,\dd x\\
	&= \int_0^1 m_\beta(gs - h)\,\dd s
	= \bigl(\pressure_\beta(g-h) - \pressure_\beta(-h)\bigr)/g.
\end{align*}
where we used Theorem~\ref{thm:PhaseTransition} in the last equality.
Therefore,
\[
	\bar{h} = \sup\setof{h\in\R}{\pressure_\beta(g-h) - \pressure_\beta(-h) \geq mg},
\]
which implies, by strict convexity of the pressure \(\pressure_\beta\), that \(\bar h\) is the unique solution to the equation \(\pressure_\beta(g-h) - \pressure_\beta(-h) = mg\).

When \(\beta\leq\betac\) or when \(\beta>\betac\) and \(\bar{h}/g \notin [0,1)\), the limiting macroscopic profile
\[
	q^*(x) = m_\beta( gx_d - \bar{h} )
\]
is an analytic function of the height \(x_d\in [0,1)\). When \(\beta>\betac\) and \(\bar{h}/g \in [0,1)\), the limiting profile is only well defined for \(x\in[0,1)^d\) such that \(x_d\neq \bar{h}/g\) and is still given by the same expression. In particular, it is still analytic in \(x_d\) on \([0,1)\setminus\{\bar{h}/g\}\), but is discontinuous when \(x_d=\bar{h}/g\), which is the height at which the interface between the two phases is located. Moreover, the jump at the discontinuity is of size \(2m_\beta^*\).

Of course, since there is no explicit expression for \(\pressure_\beta\) (except in the trivial one-dimensional case or in perturbative regimes) one cannot say much more. One example in which, one can determine the height of the interface is when \(m=0\). Indeed, \(\pressure_\beta\) being an even function, one concludes that \(\bar{h} = g/2\) and thus that the interface is located at height \(1/2\), as expected.

\subsection{\(1/N\) scaling: interface fluctuations in two dimensions}
In the continuum limit, the interface is given by a horizontal straight line. It would be of interest to understand its fluctuations for finite values of \(N\). Such an analysis might be doable by extending the techniques developed in~\cite{Ioffe+Ott+Shlosman+Velenik-2022} for a similar, albeit somewhat simpler, problem; we hope to come back to this issue in a future work. Here, we will only analyze a toy model for this interface that allows one to conjecture both the size of the typical fluctuations of the interface and its scaling limit.

To motivate the effective model, we go through the standard ``approximation by random walk'' of an Ising interface. To simplify the discussion, we only consider the canonical ensemble with \(0\) total magnetization. The probability of a given interface \(\gamma\) contains three contributions: 1) the weight of the ``free interface'' (without constraint on the total magnetization and without magnetic field), 2) the global magnetization constraint, 3) the magnetic field. We start with the following approximations.
\begin{itemize}
	\item The global magnetization constraint is replaced by \(|\gamma_+| = |\gamma_-|\) where \(\gamma_+,\gamma_-\) are the part of the box above and below \(\gamma\). This is ``justified'' by the symmetry of the system and the fact that the system relaxes quickly away from the interface.
	\item Taking as reference the flat interface, and supposing fast relaxation away from the interface, the contribution of the magnetic field to the effective action is of the form
	\begin{equation*}
		-\frac{2\beta m_{\beta}^*g}{N}\sum_{x\in\gamma} \sum_{k=1}^{|x_2|} k \ ``="\ -\frac{c(\beta, g)}{N}\sum_{x\in\gamma} (x_2)^2
	\end{equation*}
	for some suitable constant \(c(\beta,g)>0\) (in each column, the effect of having the interface at height \(t>0\) ``forces'' the spins below \(t\) to be in the minus phase, and similarity for \(t<0\)).
	\item Lastly, the free interface can be approximated by a (directed) random walk bridge with exponential tails  on the steps (this has been justified rigorously in several instances, see for example \cite{Ioffe+Ott+Shlosman+Velenik-2022}).
\end{itemize}
With these approximations (and further replacing the path of the directed random walk by the space-time trajectory of a standard one-dimensional random walk), one ends up with a random walk subject to a penalty \(\frac{c}{N}\sum_{i=1}^N S_i^2\) and to the global constraint \(\sum_{i=1}^N S_i =0\). The particular random walk model under consideration should not affect the scaling limit as long as the random walk kernel has sufficiently many moments (two should be enough, see~\cite{Hryinv+Velenik-2004,Ioffe+Shlosman+Velenik-2015} for similar considerations), we therefore study the simplest model we can think of: the one with Gaussian increments.
We are thus led to the analysis of a one-dimensional chain \(S=(S_i)_{i=0}^{N+1}\) with distribution proportional to
\begin{equation}
	 \exp \Bigl(-\sum_{i=1}^{N+1}(S_i-S_{i-1})^2 -\frac{c(\beta,g)}{N}\sum_{i=1}^{n+1} S_i^2 \Bigr) \delta(S_0)\delta(S_{N+1}) \prod_{i=1}^N\dd S_i,
\end{equation}
where \(\delta\) denotes the Dirac mass at \(0\).

\medskip
We are going to discuss three different regimes, leading to the following conjectures on the planar Ising model with Dobrushin boundary condition.

\begin{conjecture}
	In the canonical ensemble with \(m=0\) and no magnetic field, the distribution of the diffusively-rescaled interface converges to that of a Brownian bridge conditioned on having signed area \(0\).
\end{conjecture}
\begin{conjecture}
	In the canonical ensemble with \(m=0\) and magnetic field given by~\eqref{eq:hNGrav}, the distribution of the interface, rescaled by \(N^{1/2}\) horizontally and \(N^{1/4}\) vertically, converges to a stationary Ornstein--Uhlenbeck process.\\
	The same scaling limit occurs in the grand canonical ensemble.
\end{conjecture}

We hope to return to these conjectures in future works. In the meantime, we establish the corresponding claims for the Gaussian random walk model discussed above; see Section~\ref{sec:MasslessScalingLimit} and Theorem~\ref{thm:field_convergence_OU_C_and_GC} respectively.

\subsubsection{Preliminaries: covariance structure}

Our computations will rely on the following classical property of Gaussian random vectors.
Let \(X=(X_1,X_2)\) be a Gaussian vector (both \(X_1,X_2\) being vectors) with covariance and mean
\begin{equation}
	\Sigma =
	\begin{pmatrix}
		\Sigma_{11} & \Sigma_{12} \\
		\Sigma_{21} & \Sigma_{22}
	\end{pmatrix}, \quad
	\mu =
	\begin{pmatrix}
		\mu_{1} \\
		\mu_{2} 
	\end{pmatrix}.
\end{equation}
Then, if \(\Sigma_{22}\) is invertible, \(X_1\) conditioned on \(X_2=v\) is a Gaussian vector with covariance and mean
\begin{equation}
	\label{eq:Gaussian_conditional_distrib}
	\tilde{\Sigma} = \Sigma_{11} -\Sigma_{12}\Sigma_{22}^{-1}\Sigma_{21},\quad \tilde{\mu} = \mu_1+ \Sigma_{12}\Sigma_{22}^{-1}(v-\mu_2).
\end{equation}

Let us denote the system size by \(n\); we'll keep it implicit in this section.
We are interested in the massive \(1\)D Gaussian chain with Dirichlet boundary condition, with distribution
\begin{equation}
	dP_{m}(\varphi_0,\varphi_1,\dots, \varphi_{n+1}) \propto \delta(\varphi_0)\delta(\varphi_{n+1})e^{-\frac{1}{4}\sum_{i=1}^{n+1}(\varphi_i-\varphi_{i-1})^2 -\frac{m^2}{2}\sum_{i=1}^{n+1} \varphi_i^2} d\varphi,
\end{equation}
where \(\delta\) denotes the Dirac mass at \(0\).
(For convenience, we slightly changed the notation compared the effective model derived above. Moreover, from now on, \(m\) will always be used to denote the mass of this Gaussian free field, rather than the magnetization.)
This is a Gaussian vector, let
\begin{equation}
	G_m(i,j) \defby \Esp_m(\varphi_i\varphi_j).
\end{equation}
When \(m>0\), \(P_m\) is simply the law of the massive GFF on \(\Z\) conditioned on \(\{\varphi_0=\varphi_{n+1} = 0\}\). The covariances of the massive GFF are given by (see~\cite[equation (8.55)]{Friedli+Velenik-2017})
\begin{equation}
	\label{eq_Cov_massive_GFF_inf_vol}
	\tilde{G}_m(i,j) = \frac{e^{-\nu_m |i-j|}}{\sinh(\nu_m)},
\end{equation}
where \(\nu_m \defby \ln(1+m^2 +\sqrt{2m^2 + m^4})\).

Using~\eqref{eq_Cov_massive_GFF_inf_vol} and~\eqref{eq:Gaussian_conditional_distrib}, straightforward computations lead to the following explicit expression for \(G_m(i,j)\).
\begin{lemma}
	\label{lem:massive_GFF_covariances}
	When \(m>0\), for any \(1\leq i\leq j \leq n\),
	\begin{equation}
		G_m(i,j) =  2\frac{\sinh(\nu_{m} i)\sinh(\nu_{m} (n+1-j))}{\sinh(\nu_{m})\sinh(\nu_{m} (n+1))}.
	\end{equation}
	When \(m=0\), for any \(1\leq i\leq j \leq n\),
	\begin{equation}
		G(i,j) = 2 i\Bigl(1-\frac{j}{n+1}\Bigr).
	\end{equation}
\end{lemma}
From this one deduces the following asymptotics, useful when the two vertices are far from the boundary.
\begin{lemma}
	\label{lem:massive_cov_approx}
	For any \(m>0\), for any \(1\leq i\leq j \leq n\),
	\begin{multline*}
		(1-e^{-2\nu_m i})(1-e^{-2\nu_m(n+1-j)})\frac{e^{-\nu_m (j-i)}}{\sinh(\nu_m)}
		\leq
		G_m(i,j)
		\leq
		(1-e^{-2\nu_m(n+1)})^{-1}\frac{e^{-\nu_m (j-i)}}{\sinh(\nu_m)}.
	\end{multline*}
\end{lemma}
\begin{proof}
	Use \(2\sinh(x)= e^{x}(1-e^{-2x})\) and Lemma~\ref{lem:massive_GFF_covariances}.
\end{proof}

The next step is to compute the covariance of the height variables \(\varphi_i\) with the signed area \(X\defby \sum_{i=1}^{n}\varphi_i \). This will be used to impose the canonical constraint later on. Observe that \(X\) is a centered Gaussian random variable.

\begin{lemma}
	\label{lem:massive_GFF_covariances_signed_area}
	When \(m>0\), one has
	\begin{equation}
		\Esp_m(X\varphi_i) = \sum_{j=1}^n G_m(i,j),\quad \Esp_m(X^2) = \sum_{i,j=1}^n G_m(i,j).
	\end{equation}
	When \(m=0\), one has
	\begin{equation}
		\Esp_0(X\varphi_i) = i(n+1-i),\quad \Esp_0(X^2) = \frac{n(n+1)(n+2)}{6}.
	\end{equation}
\end{lemma}
\begin{proof}
	Writing the definition of \(X\) and summing the expressions for the covariances obtained in Lemma~\ref{lem:massive_GFF_covariances}, the claims follow from the identities \(\sum_{k=1}^n k = \frac{n(n+1)}{2}\) and \(\sum_{k=1}^n k^2 = \frac{n(n+1)(2n+1)}{6}\).
\end{proof}

\subsubsection{Scaling limit: massless case}\label{sec:MasslessScalingLimit}

We first consider the case \(m=0\).
Consider the process (\(t\in [0,1]\)),
\begin{equation}
	\tilde{W}_t^{n} \defby  (1-\{tn\})\varphi_{\lfloor tn \rfloor} + \{tn\} \varphi_{\lceil tn \rceil},
\end{equation}
where \(\{x\} \defby x-\lfloor x\rfloor\) denotes the fractional part of \(x\). Define its rescaled version by
\begin{equation}
	W^n_t \defby \frac{1}{\sqrt{n}} \tilde{W}_t^n.
\end{equation}

In the ``Grand-canonical'' case, the usual invariance principle of random walk bridge towards Brownian bridge gives the convergence of \(W_t^{(n)}\). The covariances can be read directly from Lemma~\ref{lem:massive_GFF_covariances}: for \(0\leq t_1\leq t_2\leq 1\),
\begin{equation}
	\Esp_0(W_{t_1}^{n}W_{t_2}^{n}) = 2t_1(1-t_2) + \sfo_n(1).
\end{equation}
In particular, the diffusivity constant is \(\sqrt{2}\).

In the ``Canonical'' case, the condition on the area is ``macroscopic'' and we end up with a Brownian bridge conditioned on having integral \(0\) (see~\cite{Gorgens} for a construction of this process).
As this is not the main focus of the paper, we only identify the covariance structure. Finite-dimensional moments and tightness follow in a standard way (the arguments are close to those of the next section).
We emphasize that a result of this type can be deduced in much greater generality from a suitable local limit theorem such as~\cite[Proposition~4.7]{Carmona+Petrelis}.

\begin{lemma}
	For any \(0\leq t_1\leq t_2\leq 1\), 
	\begin{equation}
		\lim_{n\to\infty} \Esp_0(W_{t_1}W_{t_2}\given X=0) = 2t_1(1-t_2)- 6t_1t_2(1-t_1)(1-t_2).
	\end{equation}
\end{lemma}
\begin{proof}
	Using~\eqref{eq:Gaussian_conditional_distrib}, it follows from Lemmas~\ref{lem:massive_GFF_covariances} and~\ref{lem:massive_GFF_covariances_signed_area} that, for any \(1\leq i\leq j\leq n\),
	\begin{align*}
		\Esp_0(\varphi_{i}\varphi_{j}\given X=0)
		&= G(i,j) - \frac{\Esp_0(\varphi_i X) \Esp_0(\varphi_j X) }{\Esp_0(X^2)}\\
		&= \frac{2i(n+1-j)}{n+1} - \frac{6ij(n+1-i)(n+1-j)}{n(n+1)(n+2)}.
	\end{align*}
	The claim now easily follows.
\end{proof}

\subsubsection{Scaling limit: massive case.}
We now turn to the more interesting case \(m = m_n = \frac{g}{\sqrt{n}}\).
For simplicity, we assume \(n\) to be even. In this case, the natural scaling of \(n^{1/4}\) for the field and \(n^{1/2}\) for the space naturally gives a process indexed by \(\R\) (as will be seen when computing the covariances). In order to be able to zoom in, define for \(t\in \R\),
\begin{equation}
	\tilde{W}_t^{n} \defby  (1-\{tn^{1/2}\})\varphi_{n/2 + \lfloor tn^{1/2} \rfloor} + \{tn^{1/2}\} \varphi_{n/2 +\lceil tn^{1/2} \rceil},
\end{equation}
and
\begin{equation}
	W_{t}^{n} \defby \frac{1}{n^{1/4}} \tilde{W}_{t}^{n}.
\end{equation}

The goal is to prove the following statement.
\begin{theorem}
	\label{thm:field_convergence_OU_C_and_GC}
	Consider either the ``grand canonical'' (\(\varphi\sim P_{m_n}\)) or the canonical (\(\varphi\sim {P_{m_n}(\,\cdot \given X=0)}\)) setting. In both cases, \(W_{t}^{n}\) converges weakly, as \(n\to\infty\), to the stationary Ornstein--Uhlenbeck process with parameters \(\theta = \sqrt{2} g,\sigma^2 = 1\).
\end{theorem}

\subsubsection*{``Grand-canonical'' case}

The proof follows the standard path: we first identify the covariance structure, then we prove convergence of finite-dimensional marginals, and we finally conclude using tightness.
\begin{lemma}
	\label{lem:field_GC_Convergence_of_Covariances}
	Let \(b\in (\tfrac12,1)\). Then, uniformly over \(n^{b} \leq i\leq j\leq n-n^{b}\),
	\begin{equation}
		G_{m_n}(i,j) = \frac{\sqrt{n} e^{-\sqrt{2}g \frac{(j-i)}{\sqrt{n}}}}{\sqrt{2} g}\bigl(1+\sfO(n^{-1/2}) \bigr).
	\end{equation}
	In particular, for any \(t_1\leq t_2\in \R\),
	\begin{equation}
		\lim_{n\to\infty} \Esp_{m_n}(W_{t_1}^{n} W_{t_2}^{n}) = \frac{1}{\sqrt{2} g}e^{-\sqrt{2}g(t_2-t_1)},
	\end{equation}
	where the convergence is uniform over compact sets.
\end{lemma}
\begin{proof}
	Let \(b,i,j\) be as in the statement. First,
	\begin{equation}
		\nu_{m} = \ln\bigl(1+m^2 + \sqrt{2}m\sqrt{1+m^2/2}\bigr) = \sqrt{2}m + \sfO(m^3).
	\end{equation}
	Lemma~\ref{lem:massive_cov_approx} yields, uniformly over \(n^b \leq i\leq j \leq n-n^b\),
	\begin{equation*}
		G_{m_n}(i,j) = \frac{e^{-\nu_{m_n} (j-i)}}{\sinh(\nu_{m_n})}\bigl(1+\sfO(e^{-cn^{b-1/2}})\bigr),
	\end{equation*}
	for some \(c=c(g)>0\). Now, as \(\sinh(x) = x + \sfO(x^3)\),
	\begin{gather*}
		e^{-\nu_{m_n} (j - i)} = e^{\sfO(n^{-1/2})} e^{-\sqrt{2}g \frac{(j-i)}{\sqrt{n}}},\\
		\sinh(\nu_{m_n}) = \frac{\sqrt{2} g}{\sqrt{n}}\bigl(1+\sfO(1/n)\bigr).
	\end{gather*}
	Therefore,
	\begin{equation}
		G_{m_n}(i,j) = \frac{\sqrt{n} e^{-\sqrt{2}g \frac{(j-i)}{\sqrt{n}}}}{\sqrt{2} g}\bigl(1+\sfO(n^{-1/2}) \bigr),
	\end{equation}
	which is the first part of the claim. The second part follows immediately using the definition of \(W^n\).
\end{proof}

\begin{lemma}
	\label{lem:field_GC_Convergence_of_finite_dim_moments}
	Let \(N\geq 0\) be an integer. Let \(t_1, t_2, \dots, t_N\in \R\). Then, for any \(\lambda_1,\dots,\lambda_N\in \R\),
	\begin{equation}
		\lim_{n\to \infty} \Esp_{m_n}\bigl(e^{\rmi\sum_{k=1}^N \lambda_k W^n_{t_k}}\bigr) = \exp(-\frac{1}{2\sqrt{2}g} \sum_{k,l=1}^N \lambda_k\lambda_l e^{-\sqrt{2}g|t_l-t_k|} ),
	\end{equation}
	the convergence being uniform over compact sets.
\end{lemma}
\begin{proof}
	The claim follows from the fact that \(\varphi\) is a Gaussian vector, the definition of \(W^n\), and Lemma~\ref{lem:field_GC_Convergence_of_Covariances}.
\end{proof}

Having proved the convergence of finite-dimensional distributions. We only have to establish tightness in order to conclude the proof of the first item in Theorem~\ref{thm:field_convergence_OU_C_and_GC}. Tightness will follow from a simple moment bound on the gradients on the field.
\begin{lemma}
	\label{lem:GC_Gradient_moments_UB}
	For any \(m\geq 0\), any \(n\geq 1\) and any \(1\leq i\leq j \leq n\),
	\begin{equation}
		\Esp_{m}\bigl(|\varphi_i -\varphi_j|^4\bigr) \leq 12|i-j|^2.
	\end{equation}
	In particular, for any \( t_1\leq t_2\),
	\begin{equation}
		\Esp_{m}\bigl(|W_{t_1}^n -W_{t_2}^n|^{4}\bigr) \leq 12|t_1-t_2|^2.
	\end{equation}
\end{lemma}
\begin{proof}
	Let \(i,j\) be as in the statement. Since \(\varphi_i-\varphi_j\) is a Gaussian random variable,
	\begin{equation}
		\Esp_{m}\bigl(|\varphi_i -\varphi_j|^4\bigr) = 3\Esp_{m}\bigl(|\varphi_i -\varphi_j|^2\bigr)^2.
	\end{equation}
	Now, \(\Esp_{m}\bigl(|\varphi_i -\varphi_j|^2\bigr)\) is decreasing in \(m^2\) (can be seen by differentiating with respect to \(m^2\) and using Wick rule). So,
	\begin{equation}
		\Esp_{m}\bigl(|\varphi_i -\varphi_j|^2\bigr) \leq \Esp_{0}\bigl(|\varphi_i -\varphi_j|^2\bigr) = \frac{2(j-i)}{n+1}\big(i + (n+1-j)\big)\leq 2 (j-i),
	\end{equation}by Lemma~\ref{lem:massive_GFF_covariances}, where the inequality uses \(i-j\leq 0\).
	The second claim follows from the first one and the definition of \(W^n\).
\end{proof}
Tightness of \((W_{t}^n)_{t\in [-M,M]}, n\geq 1\) for any \(M>0\) is then implied by the tightness of \(W_{0}^n, n\geq 1\) (which follows from Lemma~\ref{lem:field_GC_Convergence_of_Covariances}) and the above Lemma; see~\cite[Theorem~12.3]{Billingsley-1968}.

\subsubsection*{``Canonical'' case}

The proof is the same as in the ``grand canonical'' setting, once we get control over the covariances of the conditioned process.

\begin{lemma}
	Let \(m_n=g/\sqrt{n}\). Let \(M>0\) and \(b\in(\tfrac12,1)\). Then, uniformly over \(n^b \leq i\leq j \leq n-n^b\) with \(|i-j|\leq Mn^{1/2}\),
	\begin{equation}
		\Esp_{m_n}(\varphi_i\varphi_j\given X=0) = G_{m_n}(i,j)\bigl(1+\sfO(n^{-1/2})\bigr).
	\end{equation}
	In particular, for any \(t_1\leq t_2\in \R\),
	\begin{equation}
		\lim_{n\to\infty} \Esp_{m_n}(W_{t_1}^{n} W_{t_2}^{n}\given X=0) = \frac{1}{\sqrt{2} g}e^{-\sqrt{2}g(t_2-t_1)},
	\end{equation}
	the convergence being uniform over compact sets.
\end{lemma}
\begin{proof}
	By~\eqref{eq:Gaussian_conditional_distrib}, for any \(m>0\),
	\begin{equation*}
		\Esp_{m}(\varphi_i\varphi_j\given X=0) = G_{m}(i,j) - \frac{\Esp_{m}(X\varphi_i)\Esp_{m}(X\varphi_j)}{\Esp_{m}(X^2)}.
	\end{equation*}
	To get the claim, we only need to lower bound the above expression (as the ratio in the right-hand side is non-negative). Using Lemmas~\ref{lem:massive_cov_approx} and~\ref{lem:field_GC_Convergence_of_Covariances}, we obtain the bounds (valid for \(n\) large enough, \(b>1/2\) and \(n^b\leq i\leq j\leq n-n^b\))
	\begin{gather*}
		\Esp_{m_n}(X\varphi_i)\leq \frac{2}{\sinh(\nu_{m_n})} \sum_{k=1}^n e^{-\nu_{m_n} |i-k|},\quad
		\Esp_{m_n}(X\varphi_j)\leq \frac{2}{\sinh(\nu_{m_n})} \sum_{k=1}^n e^{-\nu_{m_n} |j-k|},\\
		G_{m_n}(i,j)\geq \frac{e^{-\nu_{m_n} |i-j|}}{2\sinh(\nu_{m_n})},\quad
		\Esp_{m_n}(X^2)\geq \frac{1}{2\sinh(\nu_{m_n})}\sum_{k,l=n^b}^{n-n^b} e^{-\nu_{m_n}|k-l|}.
	\end{gather*}
	From these, the fact that \(\nu_{m_n}\) is of order \(n^{-1/2}\) and the bound \(\sum_{k=1}^n e^{-\nu|k-i|}\leq \frac{2}{1-e^{-\nu}}\), we obtain
	\begin{equation*}
		\frac{\Esp_{m_n}(X\varphi_i)\Esp_{m_n}(X\varphi_j)}{G_{m_n}(i,j)\Esp_{m_n}(X^2)}
		\leq \sfO(n^{-1/2}) e^{\nu_{m_n}|i-j|},
	\end{equation*}
	which implies the claim, as \(\nu_{m_n}|i-j| = \sfO(1)\) by assumption.
\end{proof}

\subsection{\(1/N^2\) scaling: equilibrium crystal shapes}
\label{sec:GFsurface}

Let us now turn our attention to another physically relevant scaling of the gravitational field: \(g_N = g/N^2\). Again, this corresponds to a slowly varying magnetic field with \(b_N=ga_N/N^{2}\) and \(a_N\) an arbitrary sequence of positive numbers that tends to infinity in such a way that \(a_N=\sfo(N)\).
Under such a scaling, the contribution of the gravitational term to the energy is \(\sfO(N^{d-1})\) and thus cannot compete with the interaction energy term \(H_0\). In fact, it is of precisely the same order as the contribution originating from the boundary condition. It is thus important, in this section, to specify more precisely which boundary condition is used (in the previous sections, we used the free boundary condition for convenience, since the latter played no role at large scales).

Given a set \(\Lambda\Subset\Zd\), the Hamiltonian of the Ising model in \(\Lambda\) with boundary condition \(\eta\in\{-1,1\}^{\Zd}\) is the function on \(\{-1,1\}^\Lambda\) defined by
\[
	H_{\Lambda}^\eta(\sigma) \defby -\sum_{\substack{\{i,j\}\subset \Lambda_N\\ \normI{i-j} = 1}} \sigma_i\sigma_j - \sum_{\substack{i\in\Lambda_N,\, j\in\bbZ^d\setminus\Lambda_N\\\normI{i-j}=1}} \sigma_i\eta_j.
\]
We will denote the corresponding (grand canonical) partition function by \(Z_{\Lambda,\beta}^\eta\).

\subsubsection{Surface tension}

The central quantity needed to describe the macroscopic geometry of phase separation is the surface tension, which we briefly introduce now.
Let \(\vec{n}\) be a unit vector in \(\Rd\). We consider the box \(\Delta_N \defby [-N/2,N/2]^d\cap\Zd\) and the boundary condition defined by \(\eta^{\vec n}_j \defby 1\) if \(j\cdot \vec n \geq 0\) and \(-1\) otherwise.

The boundary condition \(\eta^{\vec n}\) forces the presence of an interface through the system. In a continuum limit (that is, when rescaling everything by a factor \(1/N\)), this interface converges to the intersection \(\Pi^{\vec n}\) of the rescaled box \([-\tfrac12,\tfrac12]^d\) with the plane passing through \(0\) and with normal \(\vec n\); see Figure~\ref{fig:interface}.

\begin{figure}[ht]
	\centering
	\begin{tikzpicture}
		\draw[thin] (-4,2) -- (4,-2);
		\draw[->,thick] (3,-1.5) -- ++ (.25,.5);
		\node at (3.35,-1.3) {$\vec n$};
		\node[inner sep=0pt] at (0,0) {\includegraphics[width=5cm]{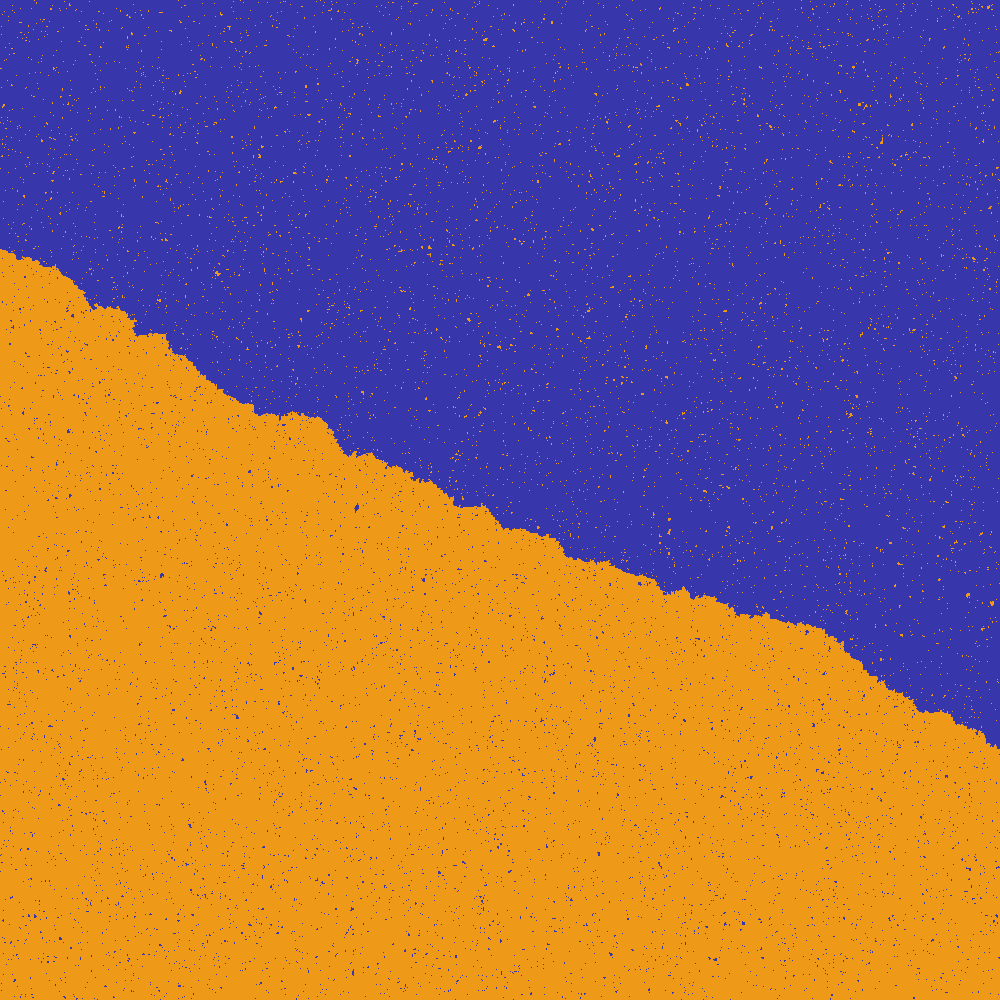}};
	\end{tikzpicture}
	\caption{A typical configuration of the planar Ising model with boundary condition \(\eta^{\vec n}\) with \(\vec n = \frac{1}{\sqrt{5}}\begin{pmatrix}1\\2\end{pmatrix}\). Here, \(+\) spins are drawn in blue and \(-\) spins in orange.}
	\label{fig:interface}
\end{figure}
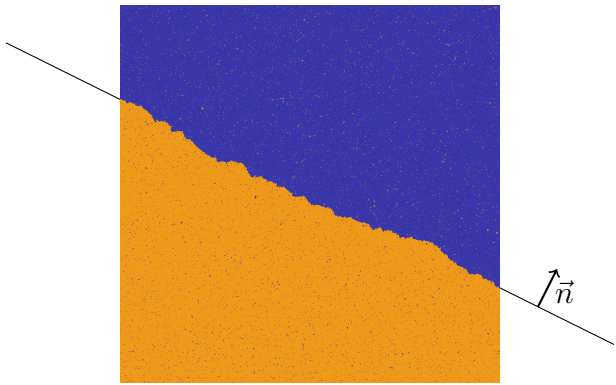

The presence of this interface contributes a correction \(-\tau_\beta(\vec n)\calH^{(d-1)}(\Pi^{\vec n})N^{d-1} + \sfo(N^{d-1})\) to the pressure of the model, where \(\tau_\beta(\vec n)\) is called the surface tension (in direction \(\vec n\)) and \(\calH^{(k)}(A)\) is the \(k\)-dimensional Hausdorff measure of the subset \(A\subset\Rd\). In other words, the surface tension is defined by
\[
	\tau_\beta(\vec n) \defby -\lim_{N\to\infty} \frac{1}{\calH^{(d-1)}(\Pi^{\vec n})N^{d-1}} \log\frac{Z^{\vec n}_{\Delta_N,\beta}}{Z_{\Delta_N,\beta}^+} ,
\]
where \(Z^{\vec n}_{\Delta_N,\beta}\) and \(Z_{\Delta_N,\beta}^+\) are the partition functions of the (grand canonical) Ising model in \(\Delta_N\) with boundary conditions \(\eta^{\vec n}\) and \(\eta^+\equiv 1\) respectively. A proof of existence of the limit can be found in~\cite{MessagerMiracleSoleRuiz}.

\subsubsection{Phase separation in the absence of a gravitational field}

\begin{figure}
	\centering
	\includegraphics[height=6cm]{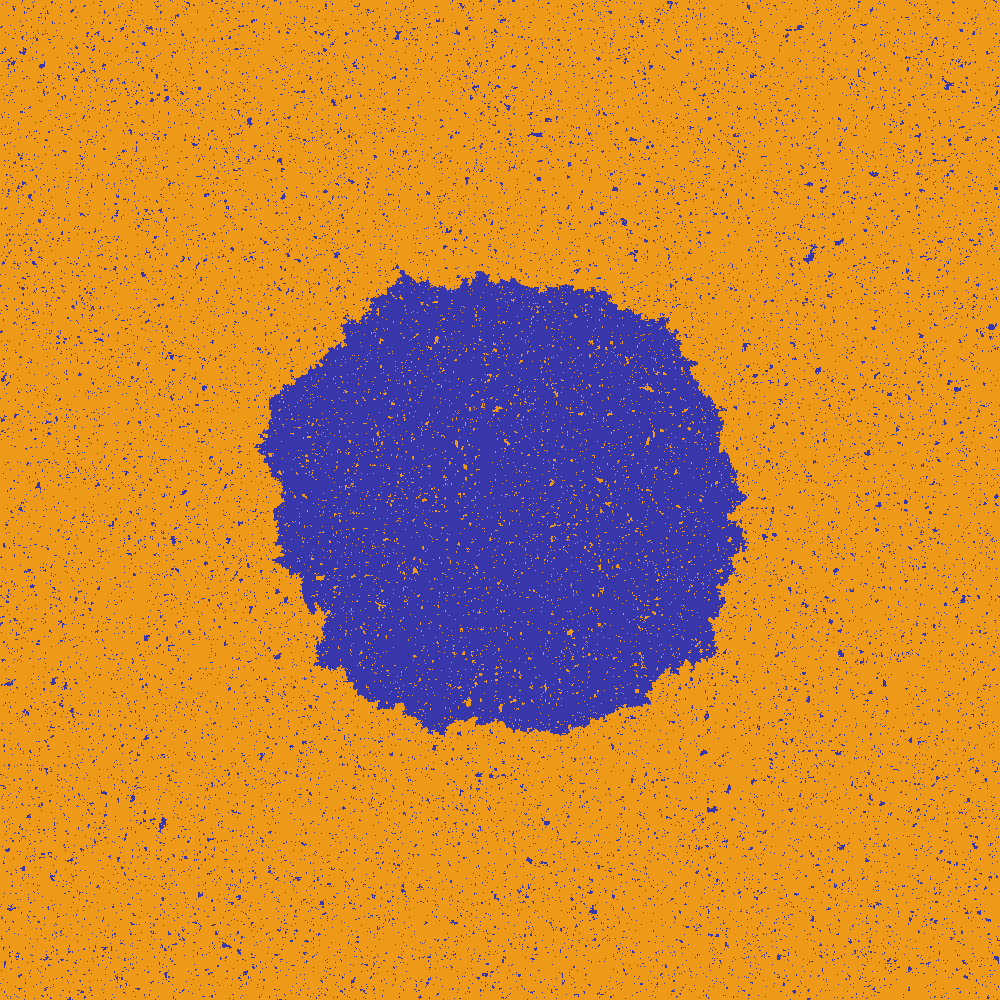}
	\caption{Typical configuration of the two-dimensional model with \(N=1000\) and orange (\(-\)) boundary condition in the absence of a gravitational field and with a fixed magnetization.}
	\label{fig:WulffNoGrav}
\end{figure}

We first recall the phenomenology in the case \(g=0\).
So, in this section, we consider the canonical Ising model in the box \(\Lambda_N\) with magnetization density \(m\), boundary condition \(\eta^- \equiv -1\) and no gravitational field (\(g=0\)).

In this case, as long as \(\abs{m}<m^*_\beta\), phase separation occurs: typical configurations contain a macroscopic droplet of one phase immersed in the other phase; see Figure~\ref{fig:WulffNoGrav}. Moreover, the shape of the droplet becomes deterministic in the continuum limit, where it is given by the solution \(\calW_\beta^m\) to the following variational problem\footnote{Of course, to really make sense, one should impose some regularity on the sets \(V\); we won't go into this here and refer instead to the review~\cite{Bodineau+Ioffe+Velenik-2001}.}:
\begin{align*}
	\textit{minimize }\int_{\partial V} \tau_\beta(\vec n_s) \,\dd\calH^{(d-1)}_s \quad
	&\textit{over all subsets }V\subset[0,1]^d\\[-2.5mm]
	&\textit{with volume } \calH^{(d)}(V)=(m_\beta^* + m)/(2m_\beta^*),
\end{align*}
where \(\vec n_s\) is the outer unit normal at \(s\).
In the absence of the constraint that \(V\subset [0,1]^d\), the (unique up to translation) solution to this variational problem is given by the Wulff shape~\cite{Taylor-1978}, that is, the dilation of the convex body
\[
	\bigcap_{\vec n \in \bbS^{d-1}} \setof{x\in\Rd}{x\cdot \vec{n} \leq \tau_\beta(\vec n)}
\]
having the required volume.
In particular, this yields the shape of the droplet when \(m\) is close enough to \(-m_\beta^*\) that a suitable translate of this solution fits inside the box \([0,1]^d\). For larger values of \(m\), the solution can still be determined explicitly, at least in dimension \(2\)~\cite{Schonmann+Shlosman-1996}, but we won't need this for our discussion here.
A rigorous derivation of this variational problem from a probabilistic analysis of the Ising model at any \(\beta>\betac\) and in any dimension \(d\geq 2\) was developed in the 1990s (in the two-dimensional setting) and the early 2000s (for higher dimensions).
Stated slightly informally, it has been proved that\footnote{Note that, while the actual statements in dimensions \(3\) and higher are indeed formulated in terms of mesoscopic profiles, much more precise statements exist in dimension \(2\). We refer to the review~\cite{Bodineau+Ioffe+Velenik-2001} and references therein.}, for any \(\beta>\betac\),
\[
	\forall\epsilon>0,\qquad \lim_{N\to\infty} \mu^-_{N,\beta,0,m}\Bigl( \inf_{V_*} \sum_{x\in\Gamma_N} \babs{\mathsf{m}_N(x) - q_{V_*}(x)} \leq \epsilon \abs{\Gamma_N} \Bigr) = 1,
\]
where the infimum is taken over all minimizers \(V_*\) of the variational problem and, for each \(x\in\Gamma_N\),
\[
	q_{V_*}(x) \defby
	\begin{cases*}
		\phantom{-}m^*_\beta	& if \(\frac1N x \in V_*\),\\
		-m^*_\beta				& otherwise.
	\end{cases*}
\]

\subsubsection{Phase separation in the presence of a gravitational field}

\begin{figure}
	\centering
	\includegraphics[height=6cm]{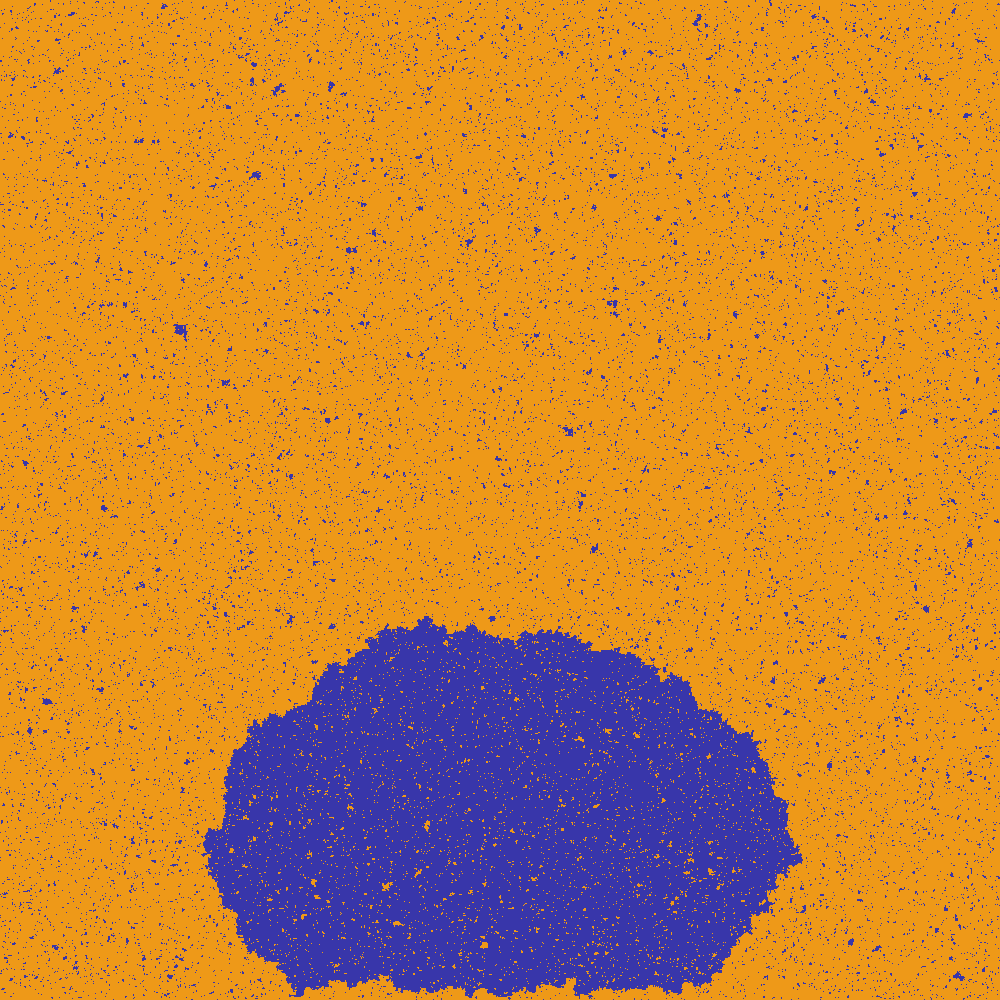}
	\hspace{5mm}
	\includegraphics[height=6cm]{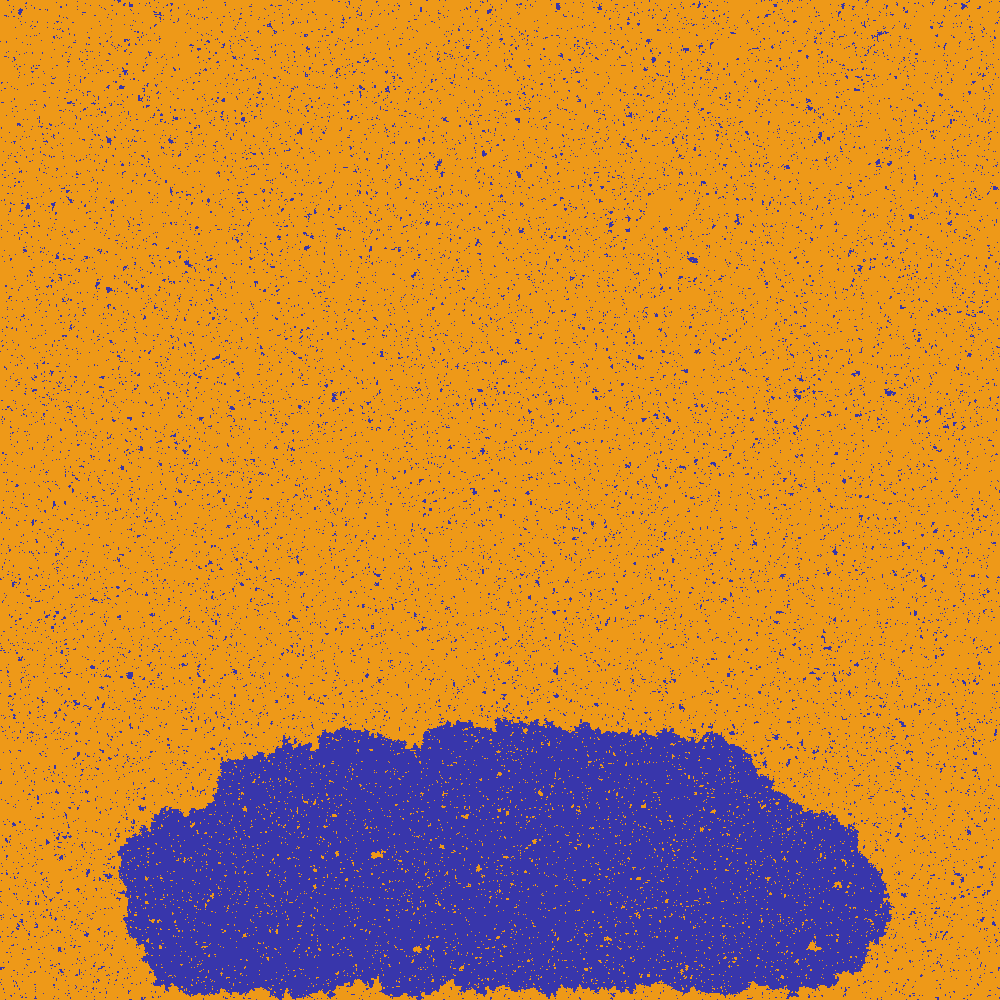}
	\caption{Typical configurations of the two-dimensional model with \(N=1000\) and orange (\(-\)) boundary condition in the presence of a gravitational field of intensity \(g_N=g/N^2\), \(g<0\), and the same total magnetization as in Figure~\ref{fig:WulffNoGrav} (left: \(g=-3\); right: \(g=-10\)). Notice how the stronger magnetic field flattens the droplet on the bottom of the box.}
	\label{fig:GFbulk-10000}
\end{figure}
Let us now consider the effect of a gravitational field of intensity \(g_N=g/N^2\), with \(g<0\), on the macroscopic geometry.
In this case, typical configurations still exhibit phase separation, but the shape of the droplet is modified by the presence of the field; see Figure~\ref{fig:GFbulk-10000}.
Of course, that the gravitational term influences the macroscopic geometry is not surprising: after all, phase separation is a surface-order phenomenon and therefore the \(\sfO(N^{d-1})\) contribution to the energy coming from the presence of the gravitational field is competing with the cost associated to the phase-separation interface. We plan to establish the following conjecture in future work.
\begin{conjecture}
	Let \(g_N=g/N^2\) and \(h_N\) be as in~\eqref{eq:hNGrav}.
	For any \(\beta>\betac\),
	\[
		\forall\epsilon>0,\qquad \lim_{N\to\infty} \mu^-_{N,\beta,h_N,m}\Bigl( \inf_{V_*} \sum_{x\in\Gamma_N} \babs{\mathsf{m}_N(x) - 	q_{V_*}(x)} \leq \epsilon \abs{\Gamma_N} \Bigr) = 1,
	\]
	where the infimum is taken over all minimizers \(V_*\) of the following variational problem:
	\begin{align*}
		\textit{minimize }\int_{\partial V} \tau_\beta(\vec n_s) \,\dd\calH^{(d-1)}_s  + 2m_\beta^*g\int_V x_d \,\dd x\quad
		&\textit{over all subsets }V\subset[0,1]^d \textit{ with}\\[-2.5mm]
		&\calH^{(d)}(V)=(m_\beta^* + m)/(2m_\beta^*).
	\end{align*}
\end{conjecture}

The solution to this variational problem is known when \(d=2\) and the constraint \(V\subset [0,1]^2\) is dropped (or when \(m\) is sufficiently close to \(-m_\beta^*\) for this solution to fit inside the box); see~\cite{Avron+Taylor+Zia-1983}. Let us note that the higher-dimensional problem is still unsolved in general.

\section*{Acknowledgments}
YA and YV are supported by the Swiss NSF grant 200021\_200422. SO is supported by the Swiss NSF grant 200021\_182237. All authors are members of the NCCR SwissMAP.

\section*{Data availability}
We do not analyze or generate any datasets, because our work is of a purely mathematical nature.

\section*{Declarations}
\noindent
\textbf{Conflict of interest:} The authors have no relevant financial or non-financial interest to disclose.

\bibliographystyle{plain}

\end{document}